\documentclass[12pt]{amsart}
\usepackage{fullpage,url,amssymb,enumerate,colonequals}
\usepackage{graphics} 
\usepackage{mathrsfs} 

\usepackage{color}

\newcommand{\defi}[1]{\textsf{\color{blue}#1}} 

\newcommand{\newproblem}[4]{
\theoremstyle{definition}
\newtheorem*{#1}{#2}
\begin{#1}
\hfill
\begin{description}
\item[input] #3
\item[question] #4
\end{description}
\end{#1}
}

\newcommand{\Aff}{\mathbb{A}}
\newcommand{\C}{\mathbb{C}}
\newcommand{\F}{\mathbb{F}}

\newcommand{\N}{\mathbb{N}}
\newcommand{\PP}{\mathbb{P}}
\newcommand{\Q}{\mathbb{Q}}
\newcommand{\R}{\mathbb{R}}
\newcommand{\Z}{\mathbb{Z}}
\newcommand{\Qbar}{{\overline{\Q}}}

\newcommand{\Fbar}{{\overline{\F}}}

\newcommand{\calA}{\mathcal{A}}

\newcommand{\calF}{\mathcal{F}}

\newcommand{\calM}{\mathcal{M}}

\newcommand{\calO}{\mathcal{O}}

\newcommand{\EE}{\mathscr{E}}

\newcommand{\tH}{{\operatorname{th}}}

\newcommand{\PGL}{\operatorname{PGL}}

\newcommand{\intersect}{\cap} 
\newcommand{\isom}{\simeq}

\newcommand{\union}{\cup} 
\newcommand{\Union}{\bigcup} 

\newtheorem*{Higman}{Higman embedding theorem}

\newtheorem{theorem}{Theorem}[section]

\newtheorem{corollary}[theorem]{Corollary}

\theoremstyle{definition}

\newtheorem{question}[theorem]{Question}

\newtheorem{example}[theorem]{Example}

\theoremstyle{remark}
\newtheorem{remark}[theorem]{Remark}

\usepackage[
	backref,
	pdfauthor={Bjorn Poonen}, 
]{hyperref}
\usepackage[alphabetic,backrefs,lite]{amsrefs} 

\begin{document}

\title{Undecidable problems: a sampler}
\author{Bjorn Poonen}
\thanks{The writing of this article was supported by the Guggenheim Foundation and National Science Foundation grants DMS-0841321 and DMS-1069236.  The final version of this survey is published as B.~Poonen, ``Undecidable problems: a sampler'' in J.~Kennedy~(ed.),
\emph{Interpreting G\"odel: Critical essays} (Cambridge University Press, 2014), 
pp.~211--241, \url{http://www.cambridge.org/gb/academic/subjects/philosophy/philosophy-science/interpreting-godel-critical-essays?format=HB}\phantom{i}.}
\address{Department of Mathematics, Massachusetts Institute of Technology, Cambridge, MA 02139-4307, USA}
\email{poonen@math.mit.edu}
\urladdr{\url{http://math.mit.edu/~poonen/}}
\date{May 28, 2012; typos corrected June 25, 2014; reference to published version added October 25, 2014.}

\begin{abstract}
After discussing two senses in which the notion of undecidability
is used, we present a survey of undecidable decision problems 
arising in various branches of mathematics.
\end{abstract}

\maketitle


\section{Introduction}\label{S:intro}

The goal of this survey article is to demonstrate that 
undecidable decision problems arise naturally in many branches of mathematics.
The criterion for selection of a problem in this survey is simply that 
the author finds it entertaining!
We do not pretend that our list of undecidable problems 
is complete in any sense.
And some of the problems we consider turn out to be decidable
or to have unknown decidability status.
For another survey of undecidable problems, see~\cite{Davis1977}.

\section{Two notions of undecidability}\label{S:undecidability}

There are two common settings in which one speaks of undecidability:
\begin{description}
\item[1. Independence from axioms] A single statement is called \defi{undecidable}
if neither it nor its negation can be deduced using the rules of logic
from the set of axioms being used.
(Example: The \defi{continuum hypothesis}, that there is no cardinal number
strictly between $\aleph_0$ and $2^{\aleph_0}$,
is undecidable in the ZFC axiom system,
assuming that ZFC itself is consistent~\cites{Godel1940,Cohen1963,Cohen1964}.)
The first examples of statements independent of a ``natural'' axiom system
were constructed by K.~G\"odel~\cite{Godel1931}.
\item[2. Decision problem] A family of problems with YES/NO answers
is called \defi{undecidable}
if there is no algorithm that terminates with the correct answer
for every problem in the family.
(Example: \defi{Hilbert's tenth problem}, to decide whether a multivariable
polynomial equation with integer coefficients has a solution in integers,
is undecidable~\cite{Matiyasevich1970}.)
\end{description}

\begin{remark}
In modern literature, the word ``undecidability'' is used more commonly
in sense~2, given that ``independence'' adequately describes sense~1.
\end{remark}

To make 2 precise, one needs a formal notion of algorithm.
Such notions were introduced by 
A.~Church~\cite{Church1936a} and A.~Turing~\cite{Turing1936} 
independently in the 1930s.
{}From now on, we interpret algorithm to mean \defi{Turing machine}, 
which, loosely speaking, means that it is a 
computer program that takes as input a finite string of 0s and 1s.
The role of the finite string 
is to specify which problem in the family is to be solved.

\begin{remark}
Often in describing a family of problems, 
it is more convenient to use higher-level mathematical objects 
such as polynomials or finite simplicial complexes as input.
This is acceptable if these objects can be encoded
as finite binary strings.  It is not necessary to specify 
the encoding as long as it is clear that a Turing machine
could convert between reasonable encodings imagined by
two different readers.
\end{remark}

\begin{remark}
One cannot speak of a \emph{single} YES/NO question 
being undecidable in sense 2, 
because there exists an algorithm that outputs the correct answer for it,
even if one might not know \emph{which} algorithm it is!
\end{remark}

There is a connection between the two notions of undecidability.
Fix a decision problem and an axiom system $\calA$ such that
\begin{enumerate}[\upshape (a)]
\item \label{I:generating A}
there is a computer program that generates exactly the axioms of $\calA$; and
\item \label{I:generating Y_i} 
there is a computer program that, when fed an instance $i$
of the decision problem, outputs a statement $Y_i$
in the language of $\calA$ such that
\begin{itemize}
\item if $Y_i$ is provable in $\calA$, 
then the answer to $i$ is YES, and
\item if $\lnot Y_i$ is provable in $\calA$, 
then the answer to $i$ is NO.
\end{itemize}
\end{enumerate}
Under these assumptions,
if the decision problem is undecidable in sense 2,
then at least one of its instance statements $Y_i$ 
is undecidable in sense 1, i.e., independent of $\calA$.
The proof of this is easy: 
if every $Y_i$ could be proved or disproved in $\calA$,
then the decision problem could be solved by a computer program that 
generates all theorems deducible from $\calA$ 
until it finds either $Y_i$ or $\lnot Y_i$.
In fact, under the same assumptions,
there must be \emph{infinitely many} $Y_i$ that are independent of $\calA$, 
since if there were only finitely many,
there would exist a decision algorithm that handled them as special cases.

\begin{remark}
In all the undecidable decision problems we present, 
the source of the undecidability can be traced back to 
a single undecidable decision problem, 
namely the halting problem, or equivalently
the membership problem for listable sets 
(see Sections \ref{S:halting problem} and~\ref{S:listable}).
For any of these problems, 
in principle we can compute a \emph{specific} $i$ for which $Y_i$ 
is independent of $\calA$
(cf.\ the last paragraph of page~294 of~\cite{Post1944}).
The value of $i$ depends on $\calA$;
more precisely, 
$i$ can be computed in terms of 
the programs in \eqref{I:generating A} and~\eqref{I:generating Y_i}.
\end{remark}

\begin{example}
Assume that ZFC is consistent, 
and, moreover, that theorems in ZFC about integers are true.
Then, because the undecidability of 
Hilbert's tenth problem in sense~2 
is proved via the halting problem (see Section~\ref{S:H10}), 
there is a specific polynomial $f \in \Z[x_1,\ldots,x_n]$
one could write down in principle such that neither
\begin{equation}
\label{E:unprovable}
	(\exists x_1,\ldots,x_n \in \Z) \; f(x_1,\ldots,x_n) = 0
\end{equation}
nor its negation can be proved in ZFC.
Moreover, \eqref{E:unprovable} must be false,
because if it were true, it could be proved in ZFC
by exhibiting a single $(x_1,\ldots,x_n) \in \Z^n$ 
satisfying $f(x_1,\ldots,x_n)=0$.
(It might seem as if this is a ZFC proof of the negation 
of~\eqref{E:unprovable}, but in fact it is only a ZFC proof 
of the \emph{implication}
\begin{quote}
   ``If ZFC is consistent and proves only true theorems
   about integers, then the negation of \eqref{E:unprovable} holds.''
\end{quote}
This observation is related to 
G\"odel's second incompleteness theorem, which implies that ZFC
cannot prove the hypothesis of the implication unless ZFC is inconsistent!)
\end{example}

\section{Logic}

G\"odel's incompleteness theorems~\cite{Godel1931}
provided undecidable statements in sense~1
for a wide variety of axiom systems.
Inspired by this,
Church and Turing began to prove that certain decision problems 
were undecidable in sense~2,
as soon as they developed their notions of algorithm.

\subsection{The halting problem}
\label{S:halting problem}

The \defi{halting problem} asks whether it is possible write a debugger
that takes as input a computer program and decides whether it eventually halts
instead of entering an infinite loop.
For convenience, let us assume that each program accepts a natural
number as input:

\newproblem{haltingpp}{Halting problem}
{a program $p$ and a natural number $x$}
{Does $p$ eventually halt when run on input $x$?}
 
\begin{theorem}[Turing~\cite{Turing1936}]
The halting problem is undecidable.
\end{theorem}

\begin{proof}[Sketch of proof]
We will use an encoding of programs as natural numbers,
and identify programs with numbers.
Suppose that there were an algorithm for deciding 
when program $p$ halts on input $x$.
Using this, we could write a new program $H$ such that
\[
	\textup{$H$ halts on input $x$} 
	\quad \iff \quad
	\textup{program $x$ does not halt on input $x$.} 
\]
Taking $x=H$, we find a contradiction: $H$ halts on input $H$
if and only if $H$ does not halt on input $H$.
\end{proof}

To turn the sketch above into a complete proof would require some programming,
to show that there is a ``universal'' computer program 
that can simulate any other program given its number;
this could then be used to construct $H$.

\subsection{Listable sets}
\label{S:listable}

Let $\N$ be the set of natural numbers.
Let $A$ be a subset of $\N$.
Call $A$ \defi{computable}\footnote{In most twentieth century literature in the subject one finds the terms \defi{recursive} and \defi{recursively enumerable (r.e.)}.  But R.~Soare~\cite{Soare1996} has argued in favor of the use of the terms ``computable'' and ``c.e.''\ instead, and many researchers in the field have followed his recommendation.}
if there is an algorithm that takes an input an element $n \in \N$
and decides whether or not $n \in A$.
On the other hand, call $A$ \defi{listable}
or \defi{computably enumerable (c.e.)}\ 
if there is a computer program that when left running forever
eventually prints out exactly the elements of $A$.
Computable sets are listable.

For each listable set $A$, we then have the following decision problem:

\newproblem{membershippp}{Membership in a listable set $A$}
{$n \in \N$}
{Is $n \in A$?}

\begin{theorem}[\cites{Church1936a,Rosser1936,Kleene1936}]
There exists a listable set $A$ for which the membership problem
is undecidable.
\end{theorem}

\begin{proof}
Let $A$ be the set of numbers of programs that halt.
Then $A$ is listable (write a program that during iteration $N$
runs each of the first $N$ programs for $N$ steps,
and prints the numbers of those that have already halted).
But the undecidability of the halting problem implies 
that $A$ is not computable; in other words,
the membership problem for $A$ is undecidable.
\end{proof}

It would be just as easy to argue in reverse,
to use the existence of a non-computable listable set
to prove the undecidability of the halting problem.

\subsection{The Entscheidungsproblem}

Fix a finite set of axioms.
Then there are some (first-order) statements that are \defi{universally valid},
meaning that they are true for every mathematical structure satisfying
the axioms.
By the completeness theorem of first-order logic~\cite{Godel1930},
the universally valid statements are exactly the ones that 
are \defi{provable} in the sense that they can be deduced from the axioms
using the rules of logic.

Can one decide in a finite amount of time
whether or not any given statement is universally valid?
This is the \defi{Entscheidungsproblem}, 
proposed by D.~Hilbert~\cite{Hilbert-Ackermann1928}*{Chapter~3, \S11}.
(Entscheidung is the German word for ``decision''.)
One could try searching for a proof by day
and searching for a proof of the negation by night,
but such an algorithm might fail to terminate for some input statements 
since it could be that neither proof exists.

More formally, but still without providing full definitions,
given a first-order logic $\calF$, possibly including a finite number of
special axioms beyond the basic axioms of first-order logic,
one has the following decision problem:

\newproblem{Entscheidugspp}{Entscheidungsproblem for $\calF$}
{a first-order sentence $s$ in the language of $\calF$}
{Is $s$ true in every model of the axioms of $\calF$?}

It was known to Hilbert that 
there is a single first-order logic $\calF_0$ without special axioms 
such that if the Entscheidungsproblem for $\calF_0$ is decidable,
so is the Entscheidungsproblem for any other first-order logic.
But Church~\cites{Church1936a,Church1936b} 
and Turing~\cite{Turing1936}*{\S11} 
independently proved that the Entscheidungsproblem for $\calF_0$
was undecidable.
For more information, see~\cite{Davis1958}*{Chapter~8,~\S4}.

\section{Combinatorics}

\subsection{The Post correspondence problem}

Imagine a rectangular block with a finite string of $a$'s and $b$'s
written along the top and another such string written along the bottom,
both upright.
When finitely many such blocks are laid side to side,
the strings along the top may be concatenated,
and the strings along the bottom may be concatenated.
E.~Post~\cite{Post1946} proved that the following 
simple-sounding problem is undecidable.

\newproblem{pcp}{Post correspondence problem}
{a finite collection of blocks, labelled as above}
{Given an unlimited supply of copies of these particular blocks, can one form a nonempty finite sequence of them for which the concatenation of the top strings equals the concatenation of the bottom strings?}

The reason that it is undecidable is that 
one can embed the halting problem in it.
Namely, with some work it is possible, given a computer program $p$,
to construct an instance of Post correspondence problem
that has a positive answer if and only if $p$ halts.

Because of its simplicity, the Post correspondence problem
is often used to prove the undecidability of other problems,
for instance, in the formal theory of languages: see~\cite{Davis1977}.

\subsection{Tiling the plane}

\defi{Wang tiles}, 
introduced by H. Wang~\cite{Wang1961}*{\S4.1},
are unit squares in the plane, with sides parallel to the axes, 
such that each side of each square has been assigned a color.
Figure~\ref{F:Wang tiles} shows a collection of $13$ such tiles.
They may be translated, but not rotated or reflected.
A tiling of the plane into such squares is valid 
if whenever two squares share an edge, the colors match,
as in the game of dominoes.
Wang proposed the following problem:

\newproblem{Wangpp}{Tiling problem}
{a finite collection of Wang tiles}
{Is there a valid tiling of the entire plane using only translated copies of the given tiles?}

Wang also conjectured~\cite{Wang1961}*{4.1.2} 
that if a tiling exists for a given finite collection,
then there exists a \emph{periodic} tiling, i.e., one that is invariant under
translations by the vectors in a finite-index subgroup of $\Z^2$,
or equivalently by the vectors in $(n\Z)^2$ 
for some fixed $n \ge 1$.
He observed that this conjecture would imply that the tiling problem
was decidable: on the $n^{\tH}$ day one could search 
for tilings that are invariant under translations in $(n\Z)^2$, 
and on the $n^{\tH}$ night one could search 
for an $n \times n$ square that cannot be tiled
(a compactness argument shows that if the entire plane cannot be tiled,
then there exists $n$ such that the $n \times n$ square cannot be tiled).

But R.~Berger~\cite{Berger1966} 
then proved that the tiling problem was undecidable,
by embedding the halting problem as a subproblem of the tiling problem.
Combining this with Wang's observation shows that there exist finite
collections that can tile the plane, but only \emph{aperiodically}.
Simplifications by R.~Robinson~\cite{Robinson1971}, J.~Kari~\cite{Kari1996}, 
and K.~Culik~II~\cite{Culik1996}
led to the example in Figure~\ref{F:Wang tiles}, with only $13$ tiles.

\begin{figure}
\resizebox{5.5in}{!}{\includegraphics{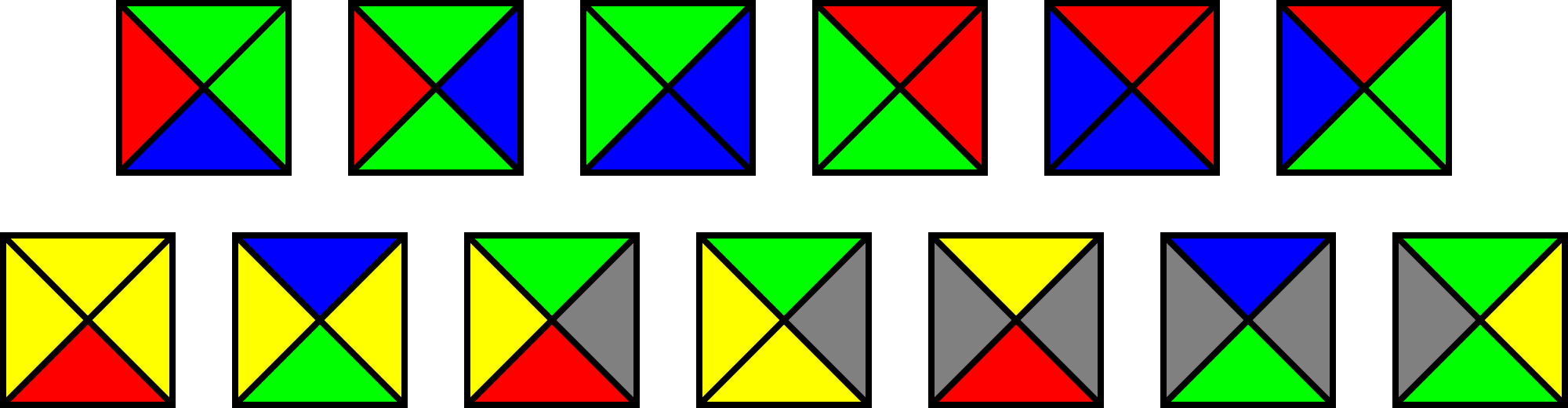}}
\caption{A collection of $13$ Wang tiles that can tile the plane, but only aperiodically.  (Source: \url{http://commons.wikimedia.org/wiki/File:Wang_tiles.svg} ,based on \cite{Culik1996}.)}
\label{F:Wang tiles}
\end{figure}

\begin{remark}
R.~Robinson~\cite{Robinson1978} and M.~Margenstern~\cite{Margenstern2008}
proved similar undecidability results for tilings of 
the \emph{hyperbolic} plane.
\end{remark}

Other tiling problems involve polyominoes.
A \defi{polyomino} is a connected planar region obtained 
by connecting finitely many unit squares along shared edges.
It is unknown whether the following is undecidable
(see~\cite{Rhoads2005}*{p.~330}, for instance):

\newproblem{polyominopp}{Polyomino tiling}
{a polyomino $P$}
{Can one tile the entire plane using translated and rotated copies of $P$?}

\subsection{Graph theory}

Fix finite graphs $G$ and $H$.
Let $V(G)$ be the vertex set of $G$; define $V(H)$ similarly.
A \defi{homomorphism} from $H$ to $G$ is a (not necessarily injective) map
$V(H) \to V(G)$ such that every edge of $H$ maps to an edge of $G$.
The \defi{homomorphism density} $t(H,G)$
is the probability that a uniformly chosen random 
map $V(H) \to V(G)$ is a homomorphism.
If $H_1 \dot{\cup} H_2$ denotes the disjoint union of graphs $H_1$ and $H_2$,
then $t(H_1 \dot{\cup} H_2,G) = t(H_1,G) t(H_2,G)$ for any $G$.

There are certain known inequalities relating these densities.
For instance, for the complete graph $K_n$ on $n$ vertices, 
elementary counting arguments similar to those in \cite{Goodman1959} 
show that 
\[
	t(K_3,G) \ge 2 t(K_2,G)^2 - t(K_2,G),
\]
or equivalently 
\[
	t(K_3,G) - 2 t(K_2 \dot{\cup} K_2,G) + t(K_2,G) \ge 0,
\]
for every finite graph $G$.
This suggests the following problem:

\newproblem{graphpp}{Linear inequalities between graph homomorphism densities}
{$k \in \Z_{\ge 0}$, finite graphs $H_1,\ldots,H_k$, and integers $a_1,\ldots,a_k$}
{Does $a_1 t(H_1,G) + \cdots + a_k t(H_k,G) \ge 0$ hold for all finite graphs $G$?}

H.~Hatami and S.~Norine proved this problem undecidable by relating it 
to Hilbert's tenth problem~\cite{Hatami-Norine2011}*{Theorem~2.12}.

\section{Matrix semigroups}

\subsection{Matrix mortality}

Given a finite list of square integer matrices,
there are many ways to form products, especially
if the factors may be repeated.
Can one decide whether some product yields the zero matrix $\mathbf{0}$?
More formally, we have the following:

\newproblem{mortalitypp}{Matrix mortality problem}
{$n \in \Z_{\ge 0}$ and a finite set $S$ of $n \times n$ integer matrices}
{Does the multiplicative semigroup generated by $S$ contain $\mathbf{0}$?}

M.~Paterson proved that this problem is undecidable,
even for sets of $3 \times 3$ matrices,
via reduction to the Post correspondence problem~\cite{Paterson1970}.
Subsequent work showed that it is undecidable also for
sets consisting of seven $3 \times 3$ 
matrices~\cite{Halava-Harju-Hirvensalo2007}*{Corollary~1}
and for sets consisting of two $21 \times 21$ 
matrices~\cite{Halava-Harju-Hirvensalo2007}*{Theorem~11}.
Whether there exists an algorithm for sets of $2 \times 2$ matrices
remains an open problem.
For a more detailed introduction to the matrix mortality problem,
see~\cite{Halava-Harju2001}.

\subsection{Freeness}

One can ask, given $n$ and $S$, 
whether distinct finite sequences of matrices in $S$
yield distinct products, i.e., whether the semigroup generated by $S$
is \defi{free}.
This turns out to be undecidable too, 
and already for sets of $3 \times 3$ 
matrices~\cite{Klarner-Birget-Satterfield1991}.
In fact, sets of fourteen $3 \times 3$ matrices
suffice for undecidability~\cite{Halava-Harju-Hirvensalo2007}*{Theorem~13}.

\subsection{Finiteness}

Can one decide whether the semigroup generated by $S$ is finite?
This time the answer turns out to be yes,
as was proved independently by G.~Jacob~\cites{Jacob1978,Jacob1977}
and by A.~Mandel and I.~Simon~\cite{Mandel-Simon1977}.

Let us outline a proof.
The main step consists of showing that there is a computable
bound $f(n,s)$ for the size of any finite semigroup of $M_n(\Z)$ 
generated by $s$ matrices.
Now for any $r \ge 1$, 
let $P_r$ be the set of products of length at most $r$ of matrices in $S$.
Start computing $P_1$, $P_2$, and so on.
If $\#P_1 < \cdots < \#P_N$ for $N=f(n,s)+1$,
then $\#P_N > f(n,s)$, so the semigroup is infinite.
Otherwise $P_r = P_{r+1}$ for some $r<N$,
in which case the semigroup equals $P_r$ and hence is finite.

The algorithm can be extended to decide finiteness
of a finitely generated semigroup of $M_n(k)$
for any finitely generated field $k$ presented 
as an explicit finite extension of $\F_p(t_1,\ldots,t_d)$ 
or $\Q(t_1,\ldots,t_d)$.

\subsection{Powers of a single matrix}

There are even some nontrivial questions 
about semigroups generated by one matrix!
Given $A \in M_k(\Z)$, can one decide whether there exists $n \in \Z_{>0}$
such that the upper right corner of $A^n$ is $0$?
This problem, whose undecidability status is unknown, is equivalent
to the following:

\newproblem{recursivesequencepp}{Zero in a linear recursive sequence}
{a linear recursive sequence of integers $(x_n)_{n \ge 0}$, 
specified by giving $x_0,\ldots,x_{k-1} \in \Z$ 
and $a_0,\ldots,a_{k-1} \in \Z$ 
such that $x_{n+k}= a_{k-1} x_{n+k-1} + \cdots + a_0 x_n$ for all $n \ge 0$}
{Does there exist $n$ such that $x_n=0$?}

This is known also as Skolem's problem, since Skolem proved
that $\{n:x_n=0\}$ is a union of a finite set and
finitely many arithmetic progressions~\cite{Skolem1934}.
See~\cite{Halava-Harju-Hirvensalo-Karhumaki-preprint}.

\section{Group theory}

Motivated by topology, M.~Dehn~\cite{Dehn1911} asked three questions
about groups:
\begin{enumerate}[\upshape 1.]
\item 
Is there an algorithm to recognize the identity of a group?
\item
Is there an algorithm to decide 
whether two given elements of a group are conjugate?
\item
Is there an algorithm to decide whether two given groups are isomorphic?
\end{enumerate}
Dehn formulated the questions precisely, 
except for the precise notion of algorithm.

\subsection{Finitely presented groups}

To make sense of such questions, one must specify
how a group is presented and how an element is presented.
A natural choice is to describe a group by means of a finite presentation
such as
\[
	S_3 = \langle r,t: r^3=1, t^2=1, trt^{-1}=r^{-1} \rangle.
\]
This example describes the group of symmetries of an equilateral triangle
as being generated by a $120^\circ$ rotation $r$ and a reflection $t$,
and lists relations satisfied by $r$ and $t$ 
such that all other relations are consequences of these.
More formally, if $n \in \Z_{\ge 0}$, and $F_n$ is the free group
on $n$ generators, and $R$ is a finite subset of $F_n$,
and $H$ is the smallest normal subgroup of $F_n$ containing $R$,
then we may form the quotient group $F_n/H$.
Any group arising in this way is called a 
\defi{finitely presented (f.p.)\ group}.
An element of an f.p.\ group can be specified by giving a \defi{word}
in the generators, i.e., a finite sequence of the generators and their
inverses, such as $rtr^{-1}r^{-1}ttt^{-1}$.

\subsection{The word problem}
\label{S:word problem}

For each fixed f.p.\ group $G$ (or more precisely, for each such group 
equipped with a particular presentation), 
we have the following:

\newproblem{wpp}{Word problem for an f.p.\ group $G$}
{word $w$ in the generators of $G$}
{Does $w$ represent $1$ in $G$?}

The decidability of the word problem
depends only on the isomorphism type of the group,
and not on the presentation.
There are many classes of groups for the word problem is decidable:
finite groups, f.p.\ abelian groups, and 
free groups on finitely many generators, for instance.
(For free groups, one algorithm is to cancel pairs 
of adjacent inverse symbols repeatedly for as long as possible; the resulting
\defi{reduced word} represents $1$ if and only if it is empty.)

But in the 1950s, P.S.\ Novikov~\cite{Novikov1955} 
and W.~Boone~\cite{Boone1959}
independently proved that there is an f.p.\ group for which 
the word problem is undecidable.
The analogue for f.p.\ semigroups had been proved earlier,
by Post~\cite{Post1947} and A.~Markov~\cites{Markov1947,Markov1951};
one proof of this goes through the undecidability
of another word problem, namely that for semi-Thue systems, 
which can also be used to prove undecidability of the 
Post correspondence problem.
Ultimately, the proofs of all these results are via reduction 
to the halting problem:
Novikov and Boone essentially showed, 
that for a certain f.p.\ group $G$, 
one could associate to any computer program $p$
a word $w$ in the generators of $G$ 
such that $w$ represents $1 \in G$ if and only if $p$ halts.

The undecidability of the word problem
admits another proof, using the Higman embedding theorem,
which we state below after introducing a definition.
A finitely generated group is called \defi{recursively presented}
if it has the form $F_n/H$, 
where $H$ is the smallest normal subgroup of $F_n$ 
containing a given subset $R$,
which is no longer required to be finite,
but is instead required to be listable.
Amazingly, it is possible to characterize such groups
without mentioning computability:
\begin{Higman}[\cite{Higman1961}]
A finitely generated group is recursively presented
if and only if
it can be embedded in a finitely presented group.
\end{Higman}

The Higman embedding theorem implies the existence
of an f.p.\ group $P$ with undecidable word problem,
as we now explain.
First, it is rather easy to construct a recursively presented group
for which the word problem is undecidable: for instance,
if $S$ is any non-computable listable set
of positive integers, then one can show that 
in the recursively presented group
\[
	G_S \colonequals \langle a,b,c,d \mid
	a^n b a^{-n} = c^n d c^{-n} \textup{ for all $n \in S$} \rangle,
\]
$a^n b a^{-n} c^n d^{-1} c^{-n}$ represents $1$ if and only if $n \in S$,
so $G_S$ has an undecidable word problem.
By the Higman embedding theorem, $G_S$ embeds in some finitely presented 
group $P$, which therefore has an undecidable word problem too.

\subsection{The conjugacy problem}

For each fixed f.p.\ group $G$, we have another problem:

\newproblem{cpp}{Conjugacy problem for an f.p.\ group $G$}
{words $w_1,w_2$ in the generators of $G$}
{Do $w_1$ and $w_2$ represent conjugate elements of $G$?}

The word problem can be viewed as the subproblem of the conjugacy problem
consisting of the instances for which $w_2$ is the empty word,
which represents $1$.
Thus the conjugacy problem for $G$ is at least as hard as
the word problem for $G$, which means that it is easier 
(or at least no harder)
to find a $G$ for which the conjugacy problem is undecidable.
In fact, P.S.~Novikov published a proof of the existence of an f.p.\ group
for which the conjugacy problem is undecidable
before publishing the result on the word problem,
and this earlier proof is much simpler~\cite{Novikov1954}.

The inequality above between the difficulties of the two problems 
is the only one, in a sense that can be made precise using basic
notions of computability theory,
namely the notions of \defi{c.e.\ degrees of unsolvability} 
and \defi{Turing reducibility} $\le_T$:

\begin{theorem}[\cite{Collins1972}]
Given c.e.\ degrees $W$ and $C$ such that $W \le_T C$,
there exists an f.p.\ group $G$ 
for which the word problem has degree $W$ 
and the conjugacy problem has degree~$C$.
\end{theorem}

This means that given c.e.\ subsets $W$ and $C$ of $\N$
such that the membership problem for $W$ is decidable 
given an oracle for the membership problem for $C$,
there exists an f.p.\ group $G$
such that the word problem for $G$ can be solved using
an oracle for membership in $W$ and vice versa,
and the conjugacy problem for $G$ can be solved using
an oracle for membership in $C$ and vice versa.

\subsection{Properties of groups}
\label{S:properties}

Instead of fixing an f.p.\ group $G$,
one can ask about algorithms that accept a finite presentation as input
and try to decide whether the group it defines has a given property.
For a wide variety of natural properties, 
the decision problem turns out to be undecidable.
To make this precise, define a \defi{Markov property}
to be a property $P$ of f.p.\ groups, depending only on the isomorphism
type of the group, not on the presentation,
such that 
\begin{enumerate}[\upshape 1.]
\item there exists an f.p.\ group $G_1$ with $P$, and
\item there exists an f.p.\ group $G_2$ that cannot be embedded in 
any f.p.\ group with $P$.
\end{enumerate}
Examples are the properties of being trivial, finite, abelian, 
nilpotent, solvable, free, or torsion-free,
because all these properties are inherited by subgroups.
The property of having a decidable word problem 
is yet another Markov property, for the same reason!

Using the undecidability of the word problem, 
S.I.~Adian~\cites{Adyan1957a,Adyan1957b} and M.~Rabin~\cite{Rabin1958}
proved the following:
\begin{theorem}
\label{T:Adian-Rabin}
For any Markov property $P$, 
it is impossible to decide whether an f.p.\ group $G$ has $P$.
\end{theorem}

\begin{corollary}
\label{C:trivial group}
It is impossible to decide whether a finite presentation
describes the trivial group.
\end{corollary}

Deciding triviality is a subproblem of the general problem 
of deciding whether two finite presentations define isomorphic groups,
so the isomorphism problem is undecidable too.

For a more extended survey of undecidability in group theory, 
see~\cite{Miller1992}.

\section{Topology}

\subsection{The homeomorphism problem}
\label{S:homeomorphism problem}

Given two manifolds, can one decide whether they are homeomorphic?
As usual, to make sense of such a question, we need to specify
how a manifold is described.  
Since every compact smooth manifold can be triangulated,
a natural choice is to use finite simplicial complexes
to represent manifolds.

\newproblem{homeopp}{Homeomorphism problem}
{finite simplicial complexes $M$ and $N$ representing smooth manifolds}
{Are $M$ and $N$ homeomorphic?}

\noindent
(One could alternatively replace homeomorphic by PL-homeomorphic,
where PL stands for piecewise-linear.)

Given a finite simplicial complex $M$ representing a compact manifold,
one obtains a subproblem
of the homeomorphism problem by fixing the first input to be $M$:

\newproblem{recognition}{Recognizing $M$}
{a finite simplicial complex $N$ representing a smooth manifold}
{Is $N$ homeomorphic to $M$?}

One can also restrict these problems according to dimension.
For $d \le 3$, the homeomorphism problem for $d$-folds turns out to be 
decidable, because of classification theorems; for $d=3$,
this uses the work of G.~Perelman on W.~Thurston's geometrization conjecture.
But for each $d \ge 4$, the homeomorphism problem for $d$-folds
is undecidable, as was proved by Markov~\cite{Markov1958}.
Moreover, S.P.\ Novikov (the son of P.S.\ Novikov!)\ 
proved that recognizing whether a finite simplicial complex
is homeomorphic to the $d$-sphere $S^d$ is an undecidable problem
for each $d \ge 5$ 
(a proof appears in the appendix to~\cite{Volodin-Kuznecov-Fomenko1974}).
{}From this, one can prove that 
for any fixed compact $d$-fold $M$ with $d \ge 5$, 
recognizing whether a finite simplicial complex is homeomorphic to $M$ 
is undecidable.

All these results are proved by reduction 
to the undecidability results for f.p.\ groups.
We now sketch the proofs of the unrecognizability results.  
(For a survey with more details, see~\cite{Weinberger2005}*{Chapter~2}.)
Fix $d \ge 5$.
Choose an f.p.\ group $G$ with undecidable word problem.
{}From $G$ and a word $w$ in the generators of $G$,
one can build an f.p.\ group $G_w$ such that $G_w$ is trivial if 
and only if $w$ represents $1$, and such that the first and second
homology groups $H_1(G_w)$ and $H_2(G_w)$ are trivial.
These conditions on $H_1$ and $H_2$ of an f.p.\ group 
are necessary and sufficient for
there to exist a \defi{homology sphere}
(a compact $d$-manifold with the same homology as $S^d$) 
with that fundamental group.
In fact, one can effectively construct a finite simplicial complex
$X_w$ representing such a homology sphere with fundamental group $G_w$.
Now: 
\begin{itemize}
\item
If $w$ represents $1$, then $G_w$ is trivial, and $X_w$ is a
simply connected homology sphere, but in dimensions $d \ge 5$
a theorem of S.~Smale~\cite{Smale1961}
implies that any such space is homeomorphic to $S^d$.
\item 
If $w$ does not represent $1$,
then $X_w$ has nontrivial fundamental group, so
$X_w$ is not homeomorphic to $S^d$.
\end{itemize}
Hence, if we had an algorithm to recognize whether a finite simplicial 
complex is homeomorphic to $S^d$,
it could be used to solve the word problem for $G$,
a contradiction.
Thus recognizing $S^d$ is an undecidable problem.

Next suppose that $M$ is {\em any} compact $d$-fold for $d \ge 5$.
The \defi{connected sum} $M \# X_w$ is obtained by punching a small
hole in each of $M$ and $X_w$ and connecting them with a thin cylinder.
This construction can be done effectively on finite simplicial complexes.
The fundamental group $\pi_1(M \# X_w)$ is the free product of
the groups $\pi_1(M)$ and $\pi_1(X_w)$.
A group-theoretic theorem states that a free product $G * H$ 
of finitely generated groups can be isomorphic to $G$ 
only if $H$ is trivial.
Now:
\begin{itemize}
\item
If $w$ represents $1$, then $X_w$ is homeomorphic to $S^d$,
and $M \# X_w$ is homeomorphic to $M$.
\item 
If $w$ does not represent $1$,
then $M \# X_w$ does not even have the same fundamental group as $M$.
\end{itemize}
Hence, if we had an algorithm to recognize $M$,
it could be used to solve the word problem for $G$,
a contradiction.

\begin{question}
Is $S^4$ recognizable?
\end{question}

\begin{remark}
P.~Seidel used similar ideas to find undecidable problems
in \emph{symplectic} geometry~\cite{Seidel2008}*{Corollary~6.8}.
\end{remark}

\subsection{Am I a manifold?}

We have seen that it is impossible to recognize whether two
manifolds represented by given finite simplicial complexes
are homeomorphic.
Even worse, one cannot even decide whether 
a finite simplicial complex represents a manifold!
In other words, the following problem is undecidable:

\newproblem{manifoldpp}{Manifold detection}
{finite simplicial complex $M$}
{Is $M$ homeomorphic to a manifold?}

Let us prove the undecidability by embedding the word problem
in this problem.
Recall that in Section~\ref{S:homeomorphism problem},
we constructed a finite simplicial complex $X_w$, in terms of a word $w$
in the generators of an f.p.\ group $G$ with unsolvable word problem, 
such that
\begin{itemize}
\item if $w$ represents $1$, 
	then $X_w$ is homeomorphic to a sphere $S^d$, and
\item if $w$ does not represent $1$, 
	then $X_w$ is a manifold with nontrivial fundamental group.
\end{itemize}
The \defi{suspension} $SX_w$ of $X_w$ 
is the simplicial complex 
whose vertices are those of $X_w$ together with two new points $a$ and $b$,
and set of faces is
$\Union_{\Delta \in X_w} \{\Delta, \Delta \union\{a\}, \Delta \union\{b\} \}$.
Geometrically, one may realize $X_w$ in a hyperplane in $\R^n$,
and $a$ and $b$ as points on either side of the hyperplane;
then $SX_w$ is the union of the line segments connecting a 
point of $\{a,b\}$ to a point of the realization of $X_w$.
Now:
\begin{itemize}
\item If $w$ represents $1$,
       then $X_w$ is homeomorphic to a sphere $S^d$,
       and $SX_w$ is homeomorphic to a sphere $S^{d+1}$.
\item If $w$ does not represent $1$, 
then $X_w$ has nontrivial fundamental group,
so $SX_w$ contains loops arbitrarily close to $a$
with nontrivial class in the fundamental group of $SX_w - \{a,b\}$,
so $SX_w$ is not locally euclidean at $a$.
\end{itemize}
Thus $SX_w$ is homeomorphic to a manifold if and only if $w$
represents $1$.
Therefore no algorithm can decide whether a given
finite simplicial complex represents a manifold.

\subsection{Knot theory}

A \defi{knot} is a smooth embedding of the circle $S^1$ in $\R^3$.
Two knots are \defi{equivalent} if there is an \defi{ambient isotopy}
that transforms one into other; loosely speaking, this means that
there is a smoothly varying family of diffeomorphisms of $\R^3$,
parametrized by an interval,
starting with the identity and ending with a diffeomorphism
that maps one knot onto the other.

How do we describe a knot in a way suitable for input into a computer?
A knot may be represented by a finite sequence of distinct points in $\Q^3$:
the knot is obtained by connecting the points in order by line segments,
the last of which connects the last point back to the first point
(we assume that each segment intersects its neighbors only at its endpoints
and intersects other segments not at all,
and the piecewise-linear curve should then be rounded at the vertices
so as to obtain a smooth curve).

\newproblem{knotpp}{Knot equivalence problem}
{knots $K_1$ and $K_2$, each represented by a finite sequence in $\Q^3$}
{Are $K_1$ and $K_2$ equivalent?}

W.~Haken constructed an algorithm to decide whether 
a knot was unknotted~\cite{Haken1961}, 
and for the general problem he outlined an approach, the last step
of which was completed by G.~Hemion~\cite{Hemion1979}.
Thus the knot equivalence problem is decidable!

One can also consider knots in higher dimension.
An \defi{$n$-dimensional knot} is a smooth embedding of $S^n$ 
in $\R^{n+2}$ (or $S^{n+2}$),
and one can define equivalence as before.
Any embedding equivalent to 
the standard embedding of $S^n$ as the unit sphere in a hyperplane
in $\R^{n+2}$ is called \defi{unknotted}.
A.~Nabutovsky and S.~Weinberger prove that 
the problem of deciding whether an $n$-dimensional knot
is unknotted is undecidable for $n \ge 3$~\cite{Nabutovsky-Weinberger1996}.
Since this is a subproblem of the equivalence problem 
for $n$-dimensional knots, the latter is undecidable too.

Nabutovsky and Weinberger leave open the following question:
\begin{question}
Is the equivalence problem for $2$-dimensional knots decidable?
\end{question}

See~\cite{Soare2004} for an exposition of some other
undecidable problems in topology and differential geometry.

\section{Number theory}

\subsection{Hilbert's tenth problem}
\label{S:H10}

One of the 23 problems in a list that Hilbert published after 
a famous lecture in 1900 asked for an algorithm to 
decide the solvability of diophantine equations:

\newproblem{h10}{Hilbert's tenth problem}
{a multivariable polynomial $f \in \Z[x_1,\ldots,x_n]$}
{Does there exist $\vec{a}\in \Z^n$ with $f(\vec{a})=0$?}

This was eventually proved undecidable by 
Yu.~Matiyasevich~\cite{Matiyasevich1970}.
To explain more, we need a definition.
Call a subset $A$ of $\Z$ \defi{diophantine} 
if there exists a polynomial $p(t,\vec{x}) \in \Z[t,x_1,\ldots,x_n]$
such that
\[
	A = \{ a \in \Z : (\exists \vec{x} \in \Z^n) \; p(a,\vec{x})=0\}.
\]
In other words, if one views $p(t,\vec{x})=0$ 
as a family of diophantine equations in the variables $x_1,\ldots,x_n$
depending on a parameter $t$, 
then $A$ is the set of parameter values
that yield a solvable diophantine equation.

It is easy to see that diophantine sets are listable.
What is remarkable is that the converse holds:

\begin{theorem}[conjectured in~\cite{Davis1953}*{p.~35}, proved in~\cite{Matiyasevich1970}]
\label{T:DPRM}
A subset of $\Z$ is diophantine if and only if it is listable.
\end{theorem}

Work of M.~Davis, H.~Putnam, and J.~Robinson culminating 
in~\cite{Davis-Putnam-Robinson1961} 
proved the analogue for \defi{exponential diophantine equations},
in which polynomials are replaced by expressions built up from
integers using not only addition and multiplication, but also exponentiation.
Matiyasevich then showed how to express exponentiation 
in diophantine terms, to complete the proof of Theorem~\ref{T:DPRM}.

Theorem~\ref{T:DPRM} immediately yields a negative answer to
Hilbert's tenth problem, because there are listable subsets $A$ of $\Z$
for which there is no algorithm to decide whether a given integer
belongs to $A$ (see Section~\ref{S:listable}).
The role played by Theorem~\ref{T:DPRM} for Hilbert's tenth problem
is similar to the role
played by the Higman embedding theorem (Section~\ref{S:word problem})
for the word problem.

\subsection{Hilbert's tenth problem for other rings}

After the negative answer to Hilbert's tenth problem,
researchers turned to variants in which the ring $\Z$
is replaced by some other commutative ring,
such as $\Q$, or the ring of integers $\calO_k$ 
of a fixed number field.

\subsubsection{The field of rational numbers}
\label{S:Q}

The problem for $\Q$ is equivalent to the problem of deciding 
whether an algebraic variety over $\Q$ has a rational point,
because any variety is a finite union of affine varieties,
and any system of equations $f_1(\vec{x})=\cdots=f_m(\vec{x})=0$
is solvable over $\Q$ if and only if 
the single equation $f_1(\vec{x})^2+\cdots+f_m(\vec{x})^2=0$ is.
It is still not known whether an algorithm exists for this problem.
The notion of a subset of $\Q$ being \defi{diophantine over $\Q$}
can be defined as in the previous section,
except with all variables running over $\Q$ instead of $\Z$.
If the subset $\Z$ were diophantine over $\Q$,
then an easy reduction to Matiyasevich's theorem
would prove the undecidability of Hilbert's tenth problem for $\Q$.
J.~Koenigsmann~\cite{Koenigsmann2010-preprint}*{Corollary~2}, 
building on \cite{Poonen2009-ae},
proved that the \emph{complement} $\Q-\Z$ is diophantine over $\Q$;
a generalization to number fields was proved by J.~Park~\cite{Park-preprint}.

In hopes of finding an undecidable problem, 
one can make the problem harder, by asking for an algorithm
to decide the truth of \defi{first-order sentences}, 
such as
\[
        (\exists x) (\forall y) (\exists z) (\exists w) 
		\quad (x \cdot z + 3 = y^2) \; \vee \; \neg ( z = x + w ).
\]
Using the theory of quadratic forms over $\Q$, 
J.~Robinson~\cite{Robinson1949} 
proved that the following decision problem is undecidable:

\newproblem{firstorderpp}{Decision problem for the first-order theory of $\Q$}
{a first-order sentence $\phi$ in the language of fields}
{Is $\phi$ true when the variables run over elements of $\Q$?}

\subsubsection{Rings of integers}

Recall that a \defi{number field} is a finite extension $k$ of $\Q$,
and that the \defi{ring of integers} $\calO_k$ of $k$ 
is the set of $\alpha \in k$ satisfying $f(\alpha)=0$ 
for some monic $f(x) \in \Z[x]$.
The problem for $\calO_k$ is conjectured
to have a negative answer for each $k$~\cite{Denef-Lipshitz1978}.
This has been proved for some $k$,
namely when
$k$ is totally real~\cite{Denef1980}, 
$k$ is a quadratic extension of 
      a totally real number field~\cite{Denef-Lipshitz1978}, or 
$k$ has exactly one conjugate pair of nonreal 
      embeddings~\cites{Pheidas1988,Shlapentokh1989}.
Through arguments of the author and A.~Shlapentokh 
\citelist{\cite{Poonen2002-h10-over-Ok}*{Theorem~1} \cite{Shlapentokh2008}*{Theorem~1.9(3)}},
certain statements about ranks of elliptic curves
over number fields would imply a negative answer for every $k$,
and such statements have been proved 
by B.~Mazur and K.~Rubin~\cite{Mazur-Rubin2010}*{\S8} 
assuming a conjecture of I.~Shafarevich and J.~Tate.

For more about Hilbert's tenth problem and its variants, 
see the survey articles~\cites{Davis-Matiyasevich-Robinson1976,Mazur1994,Poonen2008-undecidability},
the books~\cites{Matiyasevich1993,H10book,Shlapentokh2007book}, 
the website~\cite{H10web},
and the movie~\cite{H10-movie}.

\section{Analysis}

\subsection{Inequalities}
\label{S:inequalties}

Given a real-valued function on $\R$ or on $\R^n$,
can one decide whether it is nonnegative everywhere?
The answer depends on the kind of functions allowed as input.

\subsubsection{Real polynomials}

For polynomials in any number of variables, 
A.~Tarski showed that the answer is yes 
(to make sense of this, one should restrict the input
to have coefficients in $\Q$
or in the field $\R \intersect \Qbar$ of real algebraic numbers,
so that the polynomial admits a finite encoding suitable for
a Turing machine).
In fact, Tarski~\cite{Tarski1951} 
gave a decision procedure, based on elimination of quantifiers
for $\R$ in the language of ordered fields, for the following 
more general problem:

\newproblem{tarskipp}{Decision problem for the first-order theory of the ordered field $\R$}
{a first-order sentence $\phi$ in the language of ordered fields}
{Is $\phi$ true when the variables run over elements of $\R$?}

\subsubsection{Adjoining the exponential function}

If one tries to extend this by allowing expressions involving also
the real exponential function, then one runs into questions
of transcendental number theory whose answer is still unknown.
For example, can one decide for which rational numbers $r,s,t$
the equation
\[
	e^{e^r} + e^s + t = 0
\]
holds?
But assuming Schanuel's conjecture~\cite{LangTranscendental}*{pp.~30--31}, 
which rules out such ``accidental identities'',
A.~Macintyre and A.~Wilkie~\cite{Macintyre-Wilkie1996} 
gave a decision algorithm
for all first-order sentences for $\R$ with exponentiation 
in addition to the usual operations and $\le$.

\begin{remark}
In contrast, for the set $\EE$ of \emph{complex} functions
built up from integers and $z$ using addition, multiplication,
and composing with $e^z$,
A.~Adler proved that it is impossible to decide
whether a finite list of functions in $\EE$ has a common zero 
in $\C$~\cite{Adler1969}*{Theorem~1}.
This can be proved by reduction to Hilbert's tenth problem,
using two observations:
\begin{enumerate}[\upshape 1.]
\item One can characterize $\Q$ in $\C$ as the set of ratios of zeros 
of $e^z-1$.
\item One can characterize $\Z$ as the set of $x \in \Q$ 
such that there exists $z \in \C$ with $e^z = 2$ 
and $e^{zx} \in \Q$.
\end{enumerate}
\end{remark}

\subsubsection{Adjoining the sine function}
\label{S:adjoining sin}

Adjoining most other transcendental functions 
leads quickly to undecidable problems.
For example, consider the following, a variant of a theorem of 
D.~Richardson:

\begin{theorem}[cf.~\cite{Richardson1968}*{\S1, Corollary to Theorem One}]
\label{T:everywhere nonnegative}
There is a polynomial $P \in \Z[t,x_1,\ldots,x_n,y_1,\ldots,y_n]$
such that for each $a \in \Z$, the real analytic function
\[
	P(a,x_1,\ldots,x_n,\sin \pi x_1,\ldots, \sin \pi x_n)
\]
on $\R^n$ is either everywhere greater than $1$,
or else assumes values less than $-1$ and values greater than $1$,
but it is impossible to decide which, given $a$.
\end{theorem}

\begin{proof}[Sketch of proof]
By Theorem~\ref{T:DPRM},
we can find a polynomial $p(t,\vec{x}) \in \Z[t,x_1,\ldots,x_n]$ 
defining a diophantine subset $A$ of $\Z$ that is not computable.
A little analysis shows that there is another polynomial 
$G \in \Z[t,x_1,\ldots,x_n]$ whose values are positive and
growing so quickly that if we define 
\[
	L_a(\vec{x}) \colonequals 
	-2 + 4 p(a,\vec{x})^2 + G(a,\vec{x}) \sum_{i=1}^n \sin^2 \pi x_i,
\]
then $L_a(\vec{x}) \le 1$
holds only in tiny neighborhoods of the integer solutions to $p(a,\vec{x})=0$.
If $a \in A$, then such integer solutions exist 
and $L_a$ takes the value $-2$ at those integer solutions
and large positive values at some points with half-integer coordinates;
otherwise, $L_a(\vec{x}) > 1$ on $\R^n$.
\end{proof}

\begin{remark}
Richardson's original statement and proof of 
Theorem~\ref{T:everywhere nonnegative}
were slightly more involved because they came
before Hilbert's tenth problem had been proved undecidable.
Richardson instead had to use the undecidability result
for exponential diophantine equations mentioned in Section~\ref{S:H10}.
\end{remark}

M.~Laczkovich~\cite{Laczkovich2003} found a variant of 
Theorem~\ref{T:everywhere nonnegative} 
letting one use $\sin x_i$ in place of $\sin \pi x_i$ for all $i$.
Also, there exist functions $h \colon \R \to \R^n$
with dense image, such as
\[
	h(x) \colonequals (x \sin x, x \sin x^3, \ldots, x \sin x^{2n-1})
\]
(this function, used by 
J.~Denef and L.~Lipshitz in~\cite{Denef-Lipshitz1989}*{Lemma~3.2}, 
is a simpler version of one used in~\cite{Richardson1968}*{\S1,~Theorem~Two}).
By composing a multivariable function with $h$,
one obtains analogues of Theorem~\ref{T:everywhere nonnegative} 
for functions of \emph{one} variable:

\begin{theorem}[cf.~\cite{Laczkovich2003}*{Theorem~1}, 
which improves upon~\cite{Richardson1968}*{Corollary to Theorem Two}]
Let $\mathscr{S}$ be the set of functions $\R \to \R$ 
built up from integers and $x$ using addition, multiplication,
and composing with $\sin$.
Then it is impossible to decide, given $f \in \mathscr{S}$,
whether $f$ is everywhere nonnegative.
Deciding whether $f$ is everywhere positive
or whether $f$ has a zero are impossible too.
\end{theorem}

For later use, we record the fact that 
there exist functions $F_a \in \mathscr{S}$, depending in a computable
way on an integer parameter $a$,
such that either $F_a(x)>1$ on $\R$, 
or else $F_a(x)$ assumes values less than $-1$ and values greater than~$1$,
but it is impossible to decide which.

\subsection{Equality of functions}

Automatic homework graders sometimes need to decide whether
two expressions define the same function.
But deciding whether $|f(\vec{x})|$ is the same function as $f(\vec{x})$
amounts to deciding whether $f(\vec{x})$ is everywhere nonnegative,
which, by Section~\ref{S:adjoining sin},
is impossible for $f \in \Z[x_1,\ldots,x_n,\sin x_1,\ldots, \sin x_n]$
or for $f \in \mathscr{S}$ 
(cf.~\cite{Richardson1968}*{\S2,~Theorem~Two}).

For further undecidability results in analysis deduced from
the negative answer to Hilbert's tenth problem,
see \cites{Adler1969,Denef-Lipshitz1989,Stallworth-Roush1997}.

\subsection{Integration}

There exists an entire function on $\C$ whose derivative is $e^{z^2}$.
But work of J.~Liouville shows 
that no such function is expressible by an elementary formula, 
in the following sense.

For a connected open subset $U$ of $\C$, let $\calM(U)$
be the field of meromorphic functions on $U$.
Say that a function $g \in \calM(U)$ 
is \defi{elementary} if it belongs to the last field $K_n$
in a tower 
$\C(z) = K_0 \subset K_1 \subset \cdots \subset K_n$
of subfields of $\calM(U)$
such that each extension $K_{i+1}$ over $K_i$
is either algebraic or obtained by adjoining to $K_i$
either $e^f$ or a branch of $\log f$ defined on $U$ 
for some $f \in K_i$.
For instance, the trigonometric functions and their inverses
on a suitable $U$ are elementary functions.

Liouville proved a general theorem~\cite{Liouville1835}*{\S VII}
that implies that 
there is no elementary antiderivative of $e^{z^2}$ on any $U$.
(Earlier, Liouville proved that certain algebraic functions,
such as $(1+x^4)^{-1/2}$, 
have no elementary antiderivative~\cite{Liouville1833}.)
See~\cite{Rosenlicht1972} for an account of Liouville's methods.

Can one decide whether a given elementary function 
has an elementary antiderivative?
Building on the work of Liouville, 
R.~Risch sketched a positive answer to a precise version 
of this question~\cite{Risch1970}.
To obtain a positive answer, the question must be formulated carefully
to avoid having to answer questions about whether a constant 
or function is identically $0$.
For example, it is not clear whether we can decide, 
given rational numbers $r,s,t$, whether $\int (e^{e^r}+e^s+t) e^{x^2} \, dx$
is an elementary function.

Risch avoids this difficulty 
by restricting attention to functions in a tower of fields
in which the constant field is an algebraically closed field
of characteristic $0$ with a specified finite transcendence basis,
and in which each successive extension in the tower 
is either an explicit algebraic extension
or an extension adjoining $\exp f$ or a branch of $\log f$
that does not change the field of constants.

\begin{remark}
If we try to generalize by allowing expressions involving
the absolute value function $|\;|$, we encounter undecidability,
as we now explain (cf.~\cite{Richardson1968}*{\S2,~Theorem~Three}).
Define
\begin{equation}
  \label{E:sigma}
	\sigma(x) \colonequals 
	\frac12 \left( \left| x \right| - \left| x-1 \right| + 1 \right) 
	=
        \begin{cases}
          0, & \textup{ if $x \le 0$} \\
	  x, & \textup{ if $0 < x < 1$} \\
	  1, & \textup{ if $x \ge 1$.}
        \end{cases}
\end{equation}
Recall the functions $F_a$ at the end of Section~\ref{S:adjoining sin}.
Then $\sigma(-F_a(x))$ is either $0$ on all of $\R$,
or it agrees with $1$ on some open interval, 
but we cannot decide which.
Thus we cannot decide whether 
$\int \sigma(-F_a(x)) e^{x^2} \, dx$ 
is an elementary function on all of $\R$.
\end{remark}

\begin{remark}
Deciding whether an improper integral converges is undecidable too,
as was observed by P.~Wang~\cite{Wang1974}.
Specifically, we cannot decide whether
\[
	\int_{-\infty}^\infty \frac{1}{(x^2+1) F_a(x)^2} \, dx
\]
converges.
\end{remark}

\subsection{Differential equations}

Consider \defi{algebraic differential equations (ADEs)}
\[
	P(x,y,y',y'',\ldots,y^{(n)})=0
\]
to be solved by a function $y$ of $x$,
where $P$ is a polynomial with integer coefficients.
Denef and Lipshitz~\cite{Denef-Lipshitz1989}*{Theorem~4.1} 
proved that the following problem is undecidable:

\newproblem{ADEpp}{Existence of solutions to algebraic differential equations}
{$P \in \Z[x,y_1,y_2,\ldots,y_n]$}
{Does $P(x,y',y'',\ldots,y^{(n)})$ admit a real analytic solution on $[0,\infty)$?}

It remains undecidable even if one restricts the input 
so as to allow only ADEs that have a unique analytic
solution in a neighborhood of $0$.
The idea of the proof is to consider a function built up using $\sin$ 
as in Section~\ref{S:adjoining sin}
for which one cannot decide whether it is either everywhere positive, 
and then to show that its reciprocal satisfies an ADE.

\begin{remark}
To avoiding having to use $\pi$ in the coefficients of $P$,
Denef and Lipschitz observed that $\tan^{-1} x$ is a function
satisfying an ADE such that $\lim_{x \to +\infty} \tan^{-1} x = \pi/2$.
An alternative would be to use the approach of~\cite{Laczkovich2003}
for eliminating $\pi$.
\end{remark}

\begin{remark}
ADEs can behave strangely in other ways too.
L.~Rubel~\cite{Rubel1981} constructed a single explicit ADE 
whose solutions approximate any continuous function:
more precisely, for any continuous functions $f \colon \R \to \R$
and $\epsilon \colon \R \to \R_{>0}$,
there exists a $C^\infty$ solution $g \colon \R \to \R$ 
to the ADE satisfying $|g(x)-f(x)| < \epsilon(x)$ for all $x \in \R$.
\end{remark}

For other results and questions concerning existence and computability 
of solutions to differential equations,
see \cites{Jaskowski1954,Adler1969,Aberth1971,Pour-El-Richards1979,Pour-El-Richards1983,Rubel1983,Denef-Lipshitz1984,Rubel1992}.

\section{Dynamical systems}

Many nonlinear dynamical systems 
are capable of simulating universal Turing machines,
and hence they provide undecidable problems.

\subsection{Dynamical systems on $\R^n$}

Call a map $\R^n \to \R^m$ \defi{affine linear} if it is 
a linear map plus a constant vector.
Call a map $f \colon \R^n \to \R^m$ \defi{piecewise affine linear}
if $\R^n$ can be partitioned into finitely many subsets $U_i$ 
each defined by a finite number of affine linear inequalities
such that $f|_{U_i}$ agrees with an affine linear map depending on $i$.
Call such a map \defi{rational} if all the coefficients of the 
affine linear polynomials involved are rational.
Given such a map $f$, let $f^k$ be its $k^{\tH}$ iterate.
C.~Moore~\cite{Moore1990} proved that the following problem is undecidable:

\newproblem{Moorepp}{Point goes to origin in finite time}
{a rational piecewise affine linear map $f \colon \R^2 \to \R^2$ and a point $\vec{a} \in \Q^2$}
{Does there exist $k$ such that $f^k(\vec{a}) = \vec{0}$?}

Similarly, 
H.~Siegelmann and E.~Sontag proved that 
neural nets can simulate a universal Turing machine:
in particular, if $\sigma \colon \R^n \to \R^n$ is the map
obtained by applying the function~\eqref{E:sigma} coordinatewise,
then there exists $n$ and a specific matrix $A \in M_n(\Z)$,
for which it is impossible to decide, 
given a starting point $\vec{a} \in \Q^n$,
whether some iterate of $\sigma(A \vec{x})$ maps $\vec{a}$ 
to $\vec{0}$~\cite{Siegelmann-Sontag1995}.

Instead of asking about the trajectory of one point,
one can ask global questions about the dynamical system,
such as whether every trajectory converges.
V.~Blondel, O.~Bournez, P.~Koiran, and J.~Tsitsiklis
prove that many such questions are undecidable 
for piecewise affine linear maps~\cite{Blondel-et-al2001}.

For further results relating dynamical systems and computability, 
see the survey article \cites{Blondel-Tsitsiklis2000}.

\subsection{Dynamical systems on the set of positive integers}
\label{S:Collatz}

There are also undecidable problems concerning dynamics of
maps $f \colon \Z_{>0} \to \Z_{>0}$ 
such as 
\[
	f(x) \colonequals
\begin{cases}
  3x+1, &\textup{ if $x$ is odd} \\
  x/2, &\textup{ if $x$ is even.}
\end{cases}
\]
The \defi{Collatz $3x+1$ problem}, which asks
whether for every $n \in \Z_{>0}$, there exists $k$
such that $f^k(n)=1$, has been open since the 1930s~\cite{Lagarias1985};
see~\cite{Lagarias2010} for a recent compilation of articles
on this subject.
The following generalization has been proved undecidable:

\newproblem{Collatzpp}{Generalized Collatz problem}
{$m \in \Z_{>0}$, $a_0,\ldots,a_{m-1},b_0,\ldots,b_{m-1} \in \Q$ such that
the function $f$ given by $f(x) = a_i x + b_i$ for $x \bmod m = i$
maps $\Z_{>0}$ to itself}
{Is it true that for every $n \in \Z_{>0}$ there exists $k$
such that $f^k(n)=1$?}

For this and related results, see the papers by
J.~Conway~\cite{Conway1972},
by S.~Kurtz and J.~Simon~\cite{Kurtz-Simon2007},
and by J.~Endrullis, C.~Grabmayer, 
and D.~Hendriks~\cite{Endrullis-Grabmayer-Hendriks2009}.

\section{Probability}

Consider a random walk on the set $(\Z_{\ge 0})^n$ 
of lattice points in the nonnegative orthant.
At each time, the walker takes a step by adding a vector in $\{-1,0,1\}^n$.
If $\vec{x}=(x_1,\ldots,x_n)$ is the current position,
the vector to add is chosen with respect to a probability distribution 
$\Lambda_S$ depending only on the set $S = \{i : x_i \ne 0\}$,
and $\Lambda_S$ is such that the walker never leaves the orthant.
Suppose also that every probability in the description of each $\Lambda_S$ 
is in $\Q$.
Say that the random walk starting at $\vec{a}$ is \defi{stable} 
if there exists $C>0$
such that with probability~$1$ the walker returns to 
$\{\vec{x} : |\vec{x}| < C\}$ infinitely often.

\newproblem{randomwalkpp}{Stability of random walks}
{$n \in \Z_{\ge 0}$, probability distributions $\Lambda_S$ as above for $S \subseteq \{1,\ldots,n\}$, and $\vec{a} \in (\Z_{\ge 0})^n$}
{Is the random walk starting at $\vec{a}$ stable?}

D.~Gamarnik~\cite{Gamarnik2002} proved that this problem is undecidable
even if all the probabilities are $0$ or $1$!
To do this, he showed that any Turing machine could be simulated by 
such a deterministic walk.

Moreover, many basic questions about the stationary distribution of a
random walk as above turn out to be undecidable, even if one assumes
that the stationary distribution exists~\cite{Gamarnik2007}.

\section{Algebraic geometry}

\subsection{Rational sections}

K.H.~Kim and F.W.~Roush proved the undecidability of 
Hilbert's tenth problem for the field $\C(t_1,t_2)$
of rational functions in two variables~\cite{Kim-Roush1992}.
(Strictly speaking, one should assume that the input
has coefficients in $\Qbar(t_1,t_2)$
instead of $\C(t_1,t_2)$, for the sake of encoding it for
input into a Turing machine, but we will ignore this subtlety
from now on.)
By the same argument as in Section~\ref{S:Q}, 
Hilbert's tenth problem for $\C(t_1,t_2)$ is equivalent to the problem
of deciding whether 
a $\C(t_1,t_2)$-variety over has a $\C(t_1,t_2)$-point
(any pair of equations $f=g=0$ can be converted to $f^2+t_1 g^2=0$).
K.~Eisentr\"ager~\cite{Eisentraeger2004}, 
using work of L.~Moret-Bailly~\cite{Moret-Bailly2005},
generalized the Kim--Roush result to
the function field of any fixed irreducible $\C$-variety $S$ 
of dimension at least $2$.
Whether Hilbert's tenth problem for $\C(t)$ 
is undecidable is an open question,
studied in J.~Koll\'ar's article~\cite{Kollar2008}.
It is also open for the function field of each other curve over $\C$.

Let us return to the Kim--Roush result.
By interpreting the ``constants'' $t_1$ and $t_2$ as variables,
one can associate to each $\C(t_1,t_2)$-variety $X$
a $\C$-variety $Y$ 
equipped with a rational map $\pi \colon Y \dashrightarrow \Aff^2$.
The $\C(t_1,t_2)$-points of $X$ correspond to
\defi{rational sections} of $\pi$,
i.e., rational maps $s \colon \Aff^2 \dashrightarrow Y$
such that $\pi \circ s$ is the identity.
This dictionary translates the Kim--Roush result 
into the undecidability of the following problem:

\newproblem{kimroushpp}{Existence of rational sections}
{a $\C$-variety $Y$ and a rational map $\pi \colon Y \dashrightarrow \Aff^2$}
{Does $\pi$ admit a rational section?}

\subsection{Automorphisms}

Using the undecidability of Hilbert's tenth problem,
one can show that it is impossible
to decide, given a variety $X$, a point $x \in X$, 
and a subvariety $Z \subset X$,
whether there exists an automorphism of $X$ mapping $x$ 
into $Z$~\cite{Poonen2011-automorphism}.
In fact, there are fixed $X$ and $x$ for which the problem
for a variable input $Z$ is undecidable.
More precisely, 
there is a smooth projective geometrically irreducible
$\Q$-variety $X$ and a point $x \in X(\Q)$ 
such that the following problem is undecidable:

\newproblem{poonenpp}{Automorphisms mapping a point into a subvariety}
{a smooth projective geometrically irreducible subvariety $Z \subset X$}
{Does there exist an automorphism of $X$ mapping $x$ into $Z$?}

Moreover, $X$ can be chosen so that all its automorphisms
over any field extension are already defined over $\Q$,
so it does not matter whether we require the automorphisms to be defined
over the base field.

On the other hand, the following question has remained open:

\begin{question}
\label{Q:automorphism}
Is there an algorithm to decide 
whether a given variety has a nontrivial automorphism?
\end{question}

Possibly related to this is the following:

\begin{question}
\label{Q:prescribed automorphism group}
Given an f.p.\ group $G$, can one effectively construct a variety $X_G$
whose automorphism group is $G$?
\end{question}

A positive answer to Question~\ref{Q:prescribed automorphism group}
would yield a negative answer to Question~\ref{Q:automorphism},
since it is impossible to decide whether an f.p.\ group is trivial
(Corollary~\ref{C:trivial group}).

\subsection{Isomorphism}

Given the undecidability of the homeomorphism problem
for manifolds, it is natural to ask for the algebraic geometry analogue:

\newproblem{varietyisomorphismpp}{Variety isomorphism problem}
{two varieties $X$ and $Y$ over $\Qbar$}
{Is $X \isom Y$?}

The question of whether this problem might be undecidable
was asked to the author by B.~Totaro in 2007.

We stated the problem over $\Qbar$, because most algebraic geometry
is done over an algebraically closed field and we wanted the input
to admit a finite description.
Alternatively, we could work over an algebraically closed field
$\overline{\Q(t_1,t_2,\ldots)}$ 
of countable transcendence degree over $\Q$; 
this would capture the essence of the problem over $\C$, 
since any pair of varieties over $\C$ may be simultaneously
defined over a finitely generated subfield of $\C$ 
and the existence of an isomorphism is unaffected by enlarging the
ground field from one algebraically closed field to another.
One could also consider other fields, such as $\Q$, $\Fbar_p$, $\F_p$,
or $\overline{\F_p(t_1,t_2,\ldots)}$.
There is no field over which it is known 
whether one can solve the variety isomorphism problem.

\subsection{Birational equivalence}

It is also unknown whether there is an algorithm
to decide whether two given varieties are birationally equivalent.
On the other hand, given an explicit rational map,
one can decide whether it is a birational map,
and whether it is an isomorphism.

\begin{remark}
For algebraic geometers:
if at least one of the varieties $X$ and $Y$ over $\Qbar$ is of general type,
then the set of birational maps $X \dashrightarrow Y$
is finite and computable (see below),
and we can decide which of these birational maps are isomorphisms,
and hence solve the variety isomorphism problem in this restricted setting.
H.~Matsumura proved that the 
birational automorphism group of a variety of general type
is finite~\cite{Matsumura1963}.  
The set of birational maps $X \dashrightarrow Y$ is either
empty or a principal homogeneous space under this group,
so it is finite too.
Let us sketch an algorithm for computing this set.
For $n=1,2,\ldots$, 
compute the maps determined by the pluricanonical linear system $|nK|$
for $X$ and $Y$ until an $n$ is found for which at least one 
of the two maps is birational onto its image.
Then the other must be too and the linear systems must have the same
dimension, say $n$, since otherwise we know already that $X \not\isom Y$.
The birational maps are then in bijection with the linear automorphisms
of $\PP^n$ mapping one canonical image to the other,
and we can find equations for the locus of these automorphisms as
a (finite) subscheme of $\PGL_{n+1}$.
\end{remark}

\section{Algebra}

\subsection{Commutative algebra}

If we restrict the variety isomorphism problem to 
the category of \emph{affine} varieties, 
we obtain, equivalently, the isomorphism problem for
finitely generated $\Qbar$-algebras.
Here each algebra can be presented as
$\Qbar[x_1,\ldots,x_n]/(f_1,\ldots,f_m)$
by specifying $n$ and explicit polynomials $f_1,\ldots,f_m$.

One can replace $\Qbar$ by other rings of constants 
(whose elements can be encoded for computer input).
For example, taking $\Z$ yields the following problem:

\newproblem{algebraisomorphismpp}{Commutative ring isomorphism problem}
{two finitely generated commutative rings $A$ and $B$}
{Is $A \isom B$?}

The undecidability status of this problem is unknown.
In fact, the status is unknown also for the isomorphism problem
for finitely generated commutative algebras 
over any fixed nonzero commutative ring (with elements encoded
such that addition and multiplication are computable).

\subsection{Noncommutative algebra}

The noncommutative analogue of the previous problem is undecidable,
as we now explain.
Let $\Z\langle x_1,\ldots,x_n \rangle$
be the noncommutative polynomial ring (free associative algebra with $1$)
in $n$ variables over $\Z$.
A (possibly noncommutative) \defi{f.p.\ $\Z$-algebra}
is the quotient of $\Z\langle x_1,\ldots,x_n \rangle$
by the $2$-sided ideal generated by a finite list of elements $f_1,\ldots,f_m$.
For an f.p.\ group $G$, 
the group ring $\Z G$ is an f.p.\ $\Z$-algebra, 
and $\Z G \isom \Z$ if and only if $G \isom \{1\}$.
So if there were an algorithm to decide whether two f.p.\ $\Z$-algebras
were isomorphic, we could use it to decide whether an f.p.\ group
is trivial, contradicting Corollary~\ref{C:trivial group}.

For other undecidable problems concerning noncommutative f.p.\ algebras,
see~\cite{Anick1985}.

\section{Games}

\subsection{Abstract games}

Given $m \ge 1$ and a computable function 
$W \colon \N^m \to \{\textup{A},\textup{B}\}$,
consider the two-player game of no chance in which
\begin{itemize}
\item the players (A and B) alternately choose natural numbers, starting with A, and ending after $m$ numbers $x_1,\ldots,x_m$ have been chosen;
\item there is perfect information (both players know the rules and
can see all previously made choices); and
\item the winner is $W(x_1,x_2,\ldots,x_m)$.
\end{itemize}
Many games can be fit into this framework.

A result of L.~Kalm\'ar~\cite{Kalmar1928},
building on work of 
E.~Zermelo~\cite{Zermelo1913} and D.~K\"onig~\cite{Konig1927},
states that exactly one of the two players has a winning strategy.
But:

\begin{theorem}
\label{T:who wins}
It is impossible to decide, given $m$ and $W$,
\emph{which} player has a winning strategy.
\end{theorem}

\begin{proof}
Given a program $p$, consider the one-move game
in which A chooses a positive integer $x_1$
and wins if $p$ halts within the first $x_1$ steps.
Player~A has a winning strategy if and only if $p$ halts,
which is undecidable.
\end{proof}

More surprising is the following result of Rabin~\cite{Rabin1957}:

\begin{theorem}
There is a three-move game in which B has a winning strategy,
but not a \emph{computable} winning strategy (i.e., there is no computable
function of $x_1$ that is a winning move $x_2$ for B).
\end{theorem}

\begin{proof}
Post~\cite{Post1944}*{\S5} proved that there exists a \defi{simple set},
i.e., a c.e.\ set $S \subset \N$ 
whose complement $\overline{S}$ is infinite but contains no infinite c.e.\ set.
Fix such an $S$.
Let $g \colon \N \to \N$ be a computable function with $g(\N)=S$.
Consider the three-move game in which A wins if and only if $x_1+x_2=g(x_3)$.

Player~B's winning strategy is to find $t \in \overline{S}$ with $t>x_1$,
and to choose $x_2=t-x_1$.
A computable winning strategy $x_2=w(x_1)$, however, 
would yield an infinite c.e.\ subset $\{x_1+w(x_1) : x_1 \in \N\}$ 
of $\overline{S}$.
\end{proof}

\begin{remark}
Using the undecidability of Hilbert's tenth problem, 
J.P.~Jones gave new proofs of these theorems 
using games in which $W$ simply evaluates a 
given polynomial at the $m$-tuple of choices to decide 
who wins~\cite{Jones1982}.
\end{remark}

R.~Hearn~\cite{Hearn-thesis} proved that 
team games with imperfect information
can be undecidable even if they have only \emph{finitely} many positions!
For an account of this work and a complexity analysis of many finite games,
see~\cite{Hearn-Demaine2009}.

\subsection{Chess}

R.~Stanley~\cite{Stanley2010mo} 
asked whether the following problem is decidable:

\newproblem{chesspp}{Infinite chess}
{A finite list of chess pieces and their starting positions on a $\Z \times \Z$ chessboard}
{Can White force mate?}

\noindent 
(For a precise specification of the rules, and for related problems,
see~\cite{Brumleve-Hamkins-Schlicht-preprint}.)

It is unknown whether this problem is decidable.
On the other hand, D.~Brumleve, J.~Hamkins, and 
P.~Schlicht~\cite{Brumleve-Hamkins-Schlicht-preprint} 
showed that one \emph{can} decide whether White can mate in $n$ moves,
if a starting configuration and $n$ are given.
This statement can be proved quickly by encoding each instance of the problem
as a first-order sentence in \defi{Presburger arithmetic},
which is the theory of $(\N;0,1,+)$.
(Presburger arithmetic, unlike the theory of $(\N;0,1,+,\cdot)$,
is decidable~\cite{Presburger1929}.)

\section{Final remarks} 

Each undecidable problem $P$ we presented 
is at least as hard as the halting problem $H$,
because the undecidability proof ultimately depended on
encoding an arbitrary instance of $H$ as an instance of $P$.
In the other direction, many of these problems could be solved
if one could decide whether a certain search terminates;
for these $P$, 
an arbitrary instance of $P$ can be encoded as an instance of $H$.
The problems for which both reductions are possible
are called \defi{c.e.\ complete},
and they are all of exactly the same difficulty
with respect to Turing reducibility.
For example, an algorithm for deciding whether a finitely
presented group is trivial could be used to decide whether multivariable
polynomial equations have integer solutions, and vice versa!

On the other hand, certain other natural problems are 
strictly harder than the halting problem.
One such problem is the generalized Collatz problem
of Section~\ref{S:Collatz}: 
see~\cites{Kurtz-Simon2007,Endrullis-Grabmayer-Hendriks2009}.

\section*{Acknowledgements} 

I thank 
Henry Cohn, 
Martin Davis,
Jeffrey Lagarias, 
and Richard Stanley 
for discussions.

\end{document}